\documentclass[12pt,reqno]{amsart}
\usepackage{amssymb}
\usepackage{amsmath}
\usepackage{amsfonts}
\usepackage{graphicx}

\usepackage[colorlinks=true]{hyperref}
\hypersetup{urlcolor=blue, citecolor=red}

\newtheorem{theorem}{Theorem}[section]

\newtheorem{proposition}[theorem]{Proposition}
\newtheorem{remark}[theorem]{Remark}

\numberwithin{equation}{section}

\newcommand{\todoAA}[1]{\vspace{5 mm}\par \noindent
\framebox{\begin{minipage}[c]{0.95 \textwidth} \color{red}AA: \it \tt #1
\end{minipage}}\vspace{5 mm}\par}

\begin{document}

\title{Un-reduction in field theory, with applications}

\author[A. Arnaudon]{A. Arnaudon}

\address{aa, ddh: Department of Mathematics, South Kensington Campus, Imperial College London, London SW7 2AZ, United Kingdom}

\email{alexis.arnaudon@imperial.ac.uk  \,\, ,\,\,
d.holm@imperial.ac.uk}

\author[M. Castrill\'{o}n L\'{o}pez]{M. Castrill\'{o}n L\'{o}pez}

\address{mcl: ICMAT(CSIC-UAM-UC3M-UCM), Dept. Geometr\'{\i}a y
Topolog\'{\i}a, Facultad de Ciencias Matem\'{a}ticas, Universidad
Complutense de Madrid, 28040, Madrid, Spain}

\email{mcastri@mat.ucm.es}

\author[D.D. Holm]{D.D. Holm}



\date{}

\begin{abstract}
The un-reduction procedure introduced previously in the
context of Mechanics is extended to covariant Field Theory. 
The new covariant un-reduction procedure is applied 
to the problem of shape matching of images 
which depend on more than one independent variable (for
instance, time and an additional labelling parameter). Other
possibilities are also explored: non-linear $\sigma$-models and the
hyperbolic flows of curves.
\end{abstract}

\maketitle

\section{Introduction}

Symmetry (i.e., invariance under a Lie group of transformations) 
greatly facilitates the study of variational problems, both for the
construction of explicit solutions of the variational equations and
for their qualitative analysis. A rich variety of information arises 
from Lie symmetry of variational problems, especially when they 
are formulated geometrically. For example, a vast,
interesting literature exists on the topic of reduction by symmetry. 
In reduction by symmetry, the idea is to take advantage of the
group of symmetry transformations to reduce the dimension of the configuration
and phase spaces of the variational problem, thereby making the problem easier to handle. 
When performing such a reduction, one must also provide a method
of reconstructing the solutions of the original, unreduced, variational problem
from solutions of the reduced problem, which sometimes requires
additional compatibility conditions.

Surprisingly, there are nice instances where this procedure can be
used backwards. For example, suppose a variational problem looks 
complicated, but it may be recognised as the reduction by a certain group of symmetries 
of a variational problem formulated in a bigger space. 
Although the dimension of the corresponding un-reduced
configuration space may be larger, the equations or the space itself
may be simpler. Furthermore, the existence of the groups of
symmetries may shed light on the nature of the initial equations.
In this situation, one should notice that reduction by symmetries
changes the structure of the equations. For example, in the
Lagrange-Poincar\'{e} reduction procedure (when the configuration
space is a manifold $Q$ on which a Lie symmetry group $G$ acts
properly, see \cite{CaRa}, \cite{CeMaRa2001}, \cite{GBRa}), the
reduced variational equations split into two different types. The first type
is an Euler-Lagrange operator coupled with a gyroscopic
term (the curvature of a chosen connection $\mathcal{A}$ in the
bundle $Q\rightarrow Q/G$). The second type is a
conservation law. In order to have a free variational problem in
the reduced space, one needs to introduce forces into the un-reduced
principle so that the equations will decouple. The
choice of this force can be made by splitting the Lagrangian 
into horizontal and vertical parts with respect to the connection
$\mathcal{A}$. This is the un-reduction construction given in
\cite{BEBH} for variational problems of a particle (Mechanics) and
generalized in this article to a covariant field theoretical
setting. In particular, we also explore the topological situations
which arise when the parameter manifold is not longer simply connected.

The main motivation of \cite{BEBH} was shape matching: given two
plane shapes $S_{1},S_{2}\in \mathrm{Sh}(\mathbb{R}^{2})$,
understood as closed curves in $\mathbb{R}^{2}$, one seeks
the optimal path of shapes joining $S_{1}$ and $S_{2}$. This problem
is analysed in \cite{CoHo},\cite{MiMu} and references
therein. The space $\mathrm{Sh}(\mathbb{R}^{2})$ is a complicated
infinite dimensional manifold. However, we have 
$\mathrm{Sh}(\mathbb{R}^{2})=Q/G$, where
$G=\mathrm{Diff}^{+}(S^{1})$, and $Q$ is the space
$\mathrm{Emb}^{+}(S^{1},\mathbb{R}^{2})$ of positively embedded
parametrizations of the circle in the plane, which is a much easier
functional space than the unparameterised planar curves in $\mathrm{Sh}(\mathbb{R}^{2})$. 
By means of conveniently chosen forces, one may use un-reduction to lift the
problem of shape matching to
$\mathrm{Emb}^{+}(S^{1},\mathbb{R}^{2})$. In this article, this
situation becomes richer. In particular, we can study
matching of shapes depending on, say, two independent variables. A
primary case is where the shapes depend on time (time
evolution) and another parameter (space evolution) labelling a set of 
subjects in a research study. This so-called spatiotemporal
analysis of shapes is a recent and active field of research. For
details, the reader may consult \cite{Dur}, \cite{Ger},
\cite{Pey}. In spatiotemporal shape analysis, there are two main approaches.
These are the time-specific and subject-specific approaches, 
indicating the variable which parameterises the
evolution in shape comparisons; either for a certain subject at a sequence of times,
or for a sequence of subjects at a certain time. 
This spatio-temporal construction is illustrated in Fig. \ref{fig:humpty}. 
Note that the $x$ and $t$ variables have interchangeable meanings.
A more complex construction is found in \cite{Dur} where the authors build a
subject-specific approach together with a time-reparametrization,
with interesting applications to the compared evolution of
\emph{Homo Sapiens Neanderthalensis }and \emph{Homo Sapiens
Sapiens}, or bonobos and apes. The methodology is meant to
couple with statistical analysis. The configuration
space of this approach is $\mathrm{Diff}(\mathbb{R}^2)$ together
with the time reparametrization in $\mathrm{Diff}(\mathbb{R})$.

\begin{figure}[htpb]
	\centering
	\includegraphics{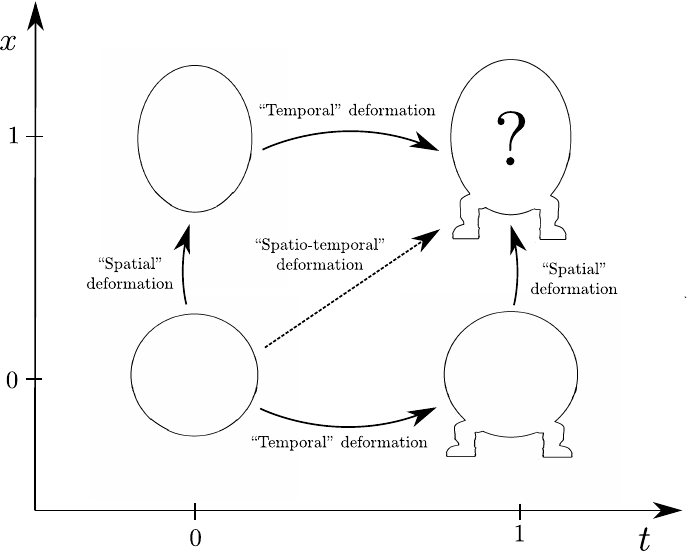}
	\caption{This diagram illustrates the spatio-temporal deformation of curves in $\mathrm{Sh}(S^1,\mathbb R^2)$ that is considered in this work. 
	The combination of spatial and temporal deformations, where the precise meaning of space and time has to be defined depending on the context, allows for a simultaneous deformation of a curve along two parameters. 
	The solution is then a function of $(x,t)$ which minimises an given energy functional. 
In the simplest case of quadratic energy functional, the solution is known as being a harmonic map. }
	\label{fig:humpty}
\end{figure}

The un-reduction procedure in $\mathrm{Emb}^+(S^1,\mathbb{R}^2)$
that we propose here provides \emph{simultaneous} evolution of both
types in a single system of partial differential equations. We
expect this combined evolution to provide more accurate and
versatile information for the problem of spatiotemporal curve
matching. Furthermore, we introduce a certain
convenient Riemannian metric in the space of embeddings depending
on derivatives of the curve (a Sobolev metric) which seems to be
appropriate for the evolution in $\mathrm{Sh}(\mathbb{R}^{2})$ and
$\mathrm{Emb}^{+}(S^{1},\mathbb{R}^{2})$ (see \cite{BBMM},
\cite{BBM}), spaces which possess some other natural but
pathological metrics. In \cite{ArCaHo}, the authors further 
investigated this approach with a simple numerical test in the classical 
mechanical setting, but more work is needed to obtain a reliable scheme.  

Because the theory is quite general, the range of potential applications is
wide. Apart from the motivation of curve matching, we point out two other
completely different areas of mathematical physics where 
covariant un-reduction is hidden. For example, $\sigma $-models in
homogeneous spaces $G/H$ may be written as an
un-reduction problem in $G$. Interestingly, we may sometimes
combine un-reduction with Euler-Poincar\'{e} reduction to the Lie
algebra $\mathfrak{g}$ to get a new set of equations. These
equations are already in the literature, but we incorporate full
geometric meaning to them with this concatenation of un-reduction
 and reduction, a situation intimately related with dual pairs (see
\cite{HoVi}). Finally, covariant un-reduction is also applied to
hyperbolic curve evolution, a baby geometric construction of other
more sophisticated geometric flow equations.

\subsection*{Plan and Main Contents of the Paper}
Section \ref{LPred-sec} reviews the basic concepts of covariant Lagrange-Poincar\'e reduction,
before formulating the main result of the paper, which is the Un-reduction Theorem \ref{mainth},
in Section \ref{unreduct-sec}.  Section \ref{apps-sec} provides examples of explicit applications
of the Un-reduction Theorem for (i) curve matching in the plane; (ii) nonlinear sigma model and 
(iii)  hyperbolic curve evolution. Each of these examples demonstrates the method of un-reduction 
and illustrates different ways to take advantage of the geometry of the reduced space. 

\section{Covariant Lagrange-Poincar\'{e} reduction}\label{LPred-sec}

The main result of the paper will be formulated as Theorem \ref{mainth} in the next section. 
This section first reviews the basic concepts of  
covariant Lagrange-Poincar\'e reduction. The version of this
reduction in Mechanics takes place when a Lie group of symmetries
$G$ acts properly on the configuration manifold $Q$ of the
variational problem under study (for example see
\cite{CeMaRa2001}). In the field theoretical setting, the group of
symmetries acts on a fibre bundle $\pi : E \to N$ by vertical
diffeomorphisms, that is, actions such that $\pi (y\cdot g)=\pi
(y), \forall y\in E, g\in G$. We refer the reader to \cite{CaRa}
and \cite{GBRa} for the exposition of the theoretical framework of
this procedure. For our purposes, in this article we have adapted
these results as follows. On one hand we just consider trivial
bundles $Q\times N \to N$, so that the dynamical objects of
interest are mappings from $N$ to $Q$ and the problem is defined
by a first order Lagrangian defined in the first jet space
$J^1(N,Q)$. This simplification is mainly done for convenience in
the applications, though the theoretical core of this work can be
done in full generalities. On the other hand, we need to incorporate
forces to our scheme, which will induce new terms in the equations in
a straightforward manner.

\subsection{Background material}

For the standard notions on bundles and connections, the reader
can go, for example, to \cite{KN}; and to \cite{Gim} for the basic
definitions on geometric variational calculus in bundles and field
theories.

Let $\pi : Q\rightarrow Q/G=\Sigma$ be a $G$-principal bundle
where the action $R_g : Q \to Q$, $g\in G$, is assumed to be on
the right. Recall that a principal connection $\mathcal{A}$ is a
$\mathfrak{g}$-valued $1$-form in $Q$ such that the equivariance
property $R_{g}^{\ast}\mathcal{A}
=\mathrm{Ad}_{g^{-1}}\circ\mathcal{A}$ holds, and
$\mathcal{A(\xi}_{Q})=\mathcal{\xi}$, for any
$\xi\in\mathfrak{g}$, where $\xi_{Q}$ is the infinitesimal
generator of the action, i.e.,
$\xi_{Q}(q):=d/d\varepsilon|_{\varepsilon=0}R_{\exp
(\varepsilon\xi)}(q)$. This definition is equivalent to a choice
of $G$-invariant splitting of the tangent bundle $TQ$ into
horizontal and vertical parts
\begin{align*}
    T_{q}Q=H_{q}Q\oplus V_{q}Q,
\end{align*}
for $q\in Q$, where $V_{q}Q=\{(\xi_{Q})_{q} :\xi\in\mathfrak{g}\}$
and $H_{q}Q=\ker\mathcal{A}$. We denote by $p^{h}:TQ\rightarrow
HQ$ and $p^{v}:TQ\rightarrow VQ$ the induced projections. The
curvature of $\mathcal{A}$ is defined to be the
$\mathfrak{g}$-valued two form $\mathcal{B}=d\mathcal{A}+[\mathcal
A,\mathcal{A}]$ and satisfies the equivariance property
$(R_g)^*\mathcal{B}=\mathrm{Ad}_{g^{-1}}\circ \mathcal{B}$. One
can also define a 2-form in $\Sigma$, but taking values in the
adjoint bundle $\mathfrak{\tilde{g}}=(Q\times\mathfrak{g})/G$ as
\begin{align*}
        \mathcal{\bar{B}}(u_{\rho},w_{\rho})=[q,\mathcal{B}(u^{h}_{q},u^{h}_{q})]_{G},\qquad
        u_{\rho},w_{\rho}\in T_{\rho} \Sigma ,
\end{align*}
where $u^{h}_q$ stands for the unique tangent vector (the
horizontal lift of $u_q$ with respect to $\mathcal{A}$) in
$H_{q}Q$ such that $T\pi(u^{h}_q)=u_\rho$. The definition does not
depend on $q\in \pi ^{-1}(\rho )$ because of the equivariant
behaviour of the curvature.

Let $N$ be an oriented manifold endowed with a volume form
$\mathbf{v}$ and consider a Lagrangian function
$L:J^{1}(N,Q)\rightarrow\mathbb{R}$ defined in the $1$-jet space
of mappings $s:N\rightarrow Q$. As the jet space $J^1 (N,Q)$ can
be naturally identified with $T^*N\otimes TQ$, we will use both
representations of this space in the following. We assume that $L$
is invariant with respect to the lifted action of $G$ in
$J^{1}(N,Q)$, defined as
\begin{align*}
    R^{(1)}_g (j^1 _x s):= j^1 _x(R_g \circ s)
\end{align*}
for $g\in G$ and any (local) mapping $s$. We can thus drop $L$ to
the quotient to obtain a reduced Lagrangian function
\begin{align*}
    \ell:J^{1}(N,Q)/G\simeq T^{\ast}N\otimes(TQ)/G\longrightarrow \mathbb{R}.
\end{align*}
If we fix a principal connection $\mathcal{A}$ of the bundle $Q\to
Q/G$, we have a diffeomorphism
\begin{align*}
        (TQ)/G  & \longrightarrow T\Sigma\oplus\mathfrak{\tilde{g}}\\
        ([v_{q}]_{G})  &\mapsto (T\pi( v_{q}),[q,\mathcal{A}(v_{q})]_G),
\end{align*}
so that the reduced phase space decomposes as
\begin{align*}
    (J^1(N,Q))/G=T^*N\otimes (TQ)/G \cong T^*N\otimes (T\Sigma \oplus
    \widetilde{\mathfrak{g}}) \cong J^1(N,\Sigma) \oplus (T^*N \otimes
    \widetilde{ \mathfrak{g}}),
\end{align*}
so that the reduced Lagrangian can then be written as
\begin{align*}
    \ell:J^{1}(N,\Sigma)\oplus(T^{\ast}N\otimes\mathfrak{\tilde{g}})\rightarrow\mathbb{R}.
\end{align*}

In the following sections, we will work with variational
principles including a force term, that is, a map
$F:J^{1}(N,Q)\rightarrow T^*Q$. The connection $\mathcal{A}$
splits the cotangent bundle $T^*Q = V^*Q \oplus H^*Q$ and we can
consider  the decomposition $F=F^{h}+F^{v}$ where
$F^{h}=p^{h}\circ F$ and $F^{v}=p^{v}\circ F$ with $p^v$ and $p^h$
denoting the projections of $V^*Q$ and $H^*Q$ respectively. We use
the same notation as for the projection of the tangent bundle as
no confusion can occur. If in addition $F$ is $G$-equivariant with
respect to the action of $G$ in both the source and target spaces,
we can drop $F^{h}$ and $F^{v}$ to $J^1(N,Q)/G$ as
\begin{align*}
        f^{h}:J^{1}(N,\Sigma)\oplus(T^*N\otimes \mathfrak{\tilde{g}})\rightarrow T^*\Sigma\ \ \mathrm{and} \ \
    f^{v}:J^{1}(N,\Sigma)\oplus(T^*N \otimes
    \mathfrak{\tilde{g})\rightarrow\tilde{g}^*}.
\end{align*}
Note that for $f^h$ we use $H^*Q/G\simeq T^* \Sigma$, and for
$f^v$ we have the isomorphism $VQ/G\simeq\mathfrak{\tilde {g}}$
given by $[(\xi_{Q})_{q}]_{G}\mapsto[q,\xi]_{G}$.

Finally, we recall the definition of the canonical momentum map
for the natural lift action of $G$ on $T^*Q$
\begin{eqnarray*}
        \mathbf J:T^*Q&\to &\mathfrak{g}^*\\
\langle \mathbf J (\alpha_q),\xi\rangle_{\mathfrak g\times\mathfrak g^*} & = &\langle \alpha_q,\xi_Q\rangle_{TQ\times TQ^*}
\end{eqnarray*}
where $\alpha_q\in T^*Q$, $\xi\in \mathfrak g$, and $\xi_Q\in TQ$.
We can extend $\mathbf{J}$ to a map
\begin{equation}
\label{momapp}
 \mathbf{J}: TN\otimes T^*Q \to
TN\otimes \mathfrak{g}^*,
\end{equation}
trivially in the factor $TN$. We note that, if we identify $TN
\simeq \wedge^{n-1}T^*N$, $n=\mathrm{dim}N$, by means of a fixed
volume form $\mathbf{v}$, the map $\mathbf{J}: TN\otimes T^*Q \to
TN\otimes \mathfrak{g}^*$ is the covariant momentum map in field
theories (cf. \cite{Gim}[Proposition 4.4]).

\subsection{Lagrange-Poincar\'e reduction}
\label{CLP}

In the sequel, we assume that $N$ is compact. If $N$ is not
compact, the domain of variations of the maps $s:N\to Q$ will be
assumed to be compactly supported. We project the variational
principle defined for $L$ from $J^1(N,Q)$ to its quotient
$J^1(N,Q)/G$. for $\ell : J^1 (N,\Sigma)\times
(T^*N\otimes\mathfrak{\tilde{g}})\to \mathbb{R}$. Critical
solutions are maps $\sigma : N \to T^*N\otimes
\mathfrak{\tilde{g}}$ which, moreover, project to maps $\rho : N
\to \Sigma = Q/G$ as $\rho = \pi_{\mathfrak{\tilde{g}}} \circ
\sigma$ according to the diagram
\begin{equation}
\label{maps}
\begin{array}{lll}
&  & \hspace{-0.21in}T^{\ast }N\otimes \mathfrak{\tilde{g}} \\
\multicolumn{1}{c}{} &
\multicolumn{1}{c}{\overset{\hspace{-0.07in}\sigma } {
\nearrow }} & \downarrow ^{\pi _{\mathfrak{\tilde{g}}}} \\
\multicolumn{1}{c}{N} & \multicolumn{1}{c}{\overset{\rho
}{\longrightarrow }} & \Sigma
\end{array}
\end{equation}
where $\pi_{\mathfrak{\tilde{g}}}:
T^{\ast}N\otimes\mathfrak{\tilde{g}}\rightarrow \Sigma$ is the
projection of the adjoint bundle forgetting the $T^*N$ factor. The
free variations of the initial problem provide a family of
constrained variations that define a new type of variational
equations. They are called Lagrange-Poincar\'e equations (see
\cite{CaRa}, \cite{GBRa}). The next theorem gives the
Lagrange-Poincar\'e reduction with forces $F$ which is obviously
the one in the literature when $F=0$.

\begin{theorem}[Covariant Lagrange-Poincar\'{e} reduction with
forces]\label{LP}
    Let \, \, $\pi : Q\rightarrow Q/G =\Sigma$ be a principal
    $G$-bundle, $\mathcal{A}$ be a principal connection on it and
    $N$ be a compact manifold oriented by a volume form $\mathbf{v}$.
    Given a map $s:N\rightarrow Q$, let $\sigma:N\rightarrow T^{\ast}N \otimes\mathfrak{\tilde{g}}$ be defined as
    \[
    \sigma(x)(\omega )= [s(x),\mathcal A(Ts\cdot (\omega))]_G,
    \]
    with $\omega\in T_xN, x\in N$; and let $\rho :N\to \Sigma$, $\rho(x) = [s(x)]_G = \pi_{\mathfrak{\tilde{g}}} \circ
    \sigma$.
    We consider a $G$-invariant Lagrangian
$L:J^{1}(N,Q)\rightarrow\mathbb{R}$ and a $G$-equivariant force
$F:J^{1}(N,Q)\rightarrow T^*Q$. Then the following points are
equivalent:

\begin{itemize}
\item[(1)] $s$ is a critical mapping of the variational principle
\begin{align}
    \delta\int_{N}L(s,j^{1}s)\mathbf{v}+\int_{N}\langle F(s,j^{1}s),\delta s\rangle\mathbf{v}=0
    \label{un-red-var}
\end{align}
with free variations $\delta s$.

\item[(2)] The Euler-Lagrange form of $L$ satisfies the relation
\begin{align*}
    \mathcal{EL}(L\mathbf{v})(j^2 s)=F.
\end{align*}

\item[(3)] $\sigma:N\rightarrow T^{\ast}N\otimes$
$\mathfrak{\tilde{g}}$
is a critical mapping of the variational principle%
\begin{align*}
        \delta\int_{N}\ell(j^{1}\rho,\sigma)\mathbf{v}+\int_N  \langle f^{h}(j^{1}\rho,\sigma),
        \delta\rho\rangle\mathbf{v}+\int_N\langle f^{v}(j^{1}\rho,\sigma),\eta\rangle \mathbf{v}=0,
\end{align*}
for variations of the form
$\delta\sigma=\nabla^{\mathcal{A}}\eta-[\sigma
,\eta]+\mathcal{\bar{B}}(\delta\rho,T\rho)\in
\tilde{\mathfrak{g}}$, where $\delta\rho\in T_\rho \Sigma$ is a
free variation of $\rho$ and $\eta$ is a free section of
$\mathfrak{\tilde{g}}\rightarrow \Sigma$.

\item[(4)] $\sigma$ satisfies the Lagrange-Poincar\'{e} equations
\begin{equation}
\left.
\begin{array}{c}
\mathcal{EL}_{\rho}(\ell\mathbf{v})=f^{h}-\left\langle \dfrac{\delta\ell}{\delta\sigma},i_{T\rho}\mathcal{\bar{B}}\right\rangle ,\\
\mathrm{div}^{\mathcal{A}}\dfrac{\delta\ell}{\delta\sigma}+\mathrm{ad}_{\sigma}^{\ast}\dfrac{\delta\ell}{\delta\sigma}=f^{v},
\end{array}
\right\}   \label{ff}
\end{equation}
where $\mathcal{EL}_{\rho}(\ell\mathbf{v}):J^{2}(N,\Sigma)\rightarrow
T^{\ast }\Sigma$ is the Euler-Lagrange form of $\ell$ with respect
to the variable $\rho$ only and $\mathrm{div}^{\mathcal{A}}$
stands for the covariant divergence operator defined by the
connection $\mathcal{A}$.
\end{itemize}
\end{theorem}

\begin{remark}
        Given a solution of the Lagrange-Poincar\'e equations \eqref{ff},
the reconstruction of a solution of the initial variational
problem requires a compatibility condition. Given the map
$\sigma:N\rightarrow T^{\ast}N\otimes \mathfrak{\tilde{g}}$ and
the induced map $\rho:N\rightarrow\Sigma$, we consider the
pull-back principal bundle $\rho^{\ast}Q\rightarrow N$ and the
pull-back of the connection $\rho^{\ast}\mathcal{A}$. Recall that
the space of connections is an affine space modelled over the
vector space of $\mathfrak{\tilde{g}}$-valued $1$-forms in the
base manifold. We can thus consider the new connection
$\mathcal{A}^{\sigma}=\rho^{\ast}\mathcal{A} +\sigma$. Then, the
compatibility condition is
\begin{equation}
\label{comp}
\mathrm{Curv}(\mathcal{A}^{\sigma})=0.
\end{equation}
Indeed, if this condition is satisfied, and the manifold $N$ is
simply connected (see \S \ref{topo} below for some topological
issues), then the solutions $s:N\rightarrow Q$ are the integral
leaves or sections of that connection. See \cite{CaRa,GBRa} for more details.
\end{remark}

\section{The covariant un-reduction scheme} \label{unreduct-sec}
We are now almost ready to describe the un-reduction scheme for
Field Theories. As in the case of Mechanics (see \cite{BEBH}),
this construction requires that the Lagrangian is decomposed into
horizontal and vertical parts with respect to the chosen
connection $\mathcal{A}$.

\subsection{Vertical and horizontal Lagrangians}

We first give an expanded expression of the Euler-Lagrange form
$\mathcal{EL}(L):J^2(N,Q)\to T^*Q$ for an arbitrary Lagrangian
$L:J^{1}(N,Q)\rightarrow \mathbb{R}$ once a linear connection
$\overline{\nabla }$ in $Q$ has been fixed. For that, we consider the
horizontal lift $v\mapsto \hat{v}$ from $TQ$ to $T(T^*N\otimes
TQ)$ with respect to $\overline{\nabla}$ (the lift is done in the $TQ$
part only and is trivial in the $T^*N$ factor). Then we define
$\frac{\overline{\nabla}L}{ds}: J^1 (N,Q)\to T^*Q$ as
\begin{align*}
    \left \langle \frac{\overline{\nabla}L}{d s}(j^1 _x s),\delta s\right \rangle_{TQ\times T^*Q}:=\mathbf dL(j^1_xs)\cdot \widehat{\delta s},
\end{align*}
for any $\delta s \in T_qQ,\,q=s(x)$.
On the other hand, we define the vertical derivative $\frac{\partial L}{\partial j^1s} : J^1(N,Q)\to
TN\otimes T^*Q$ as
\[
\left \langle \frac{\partial L}{\partial j^1s}(j^1_x s),\omega\right\rangle :=\left.\frac{d}{d\epsilon}\right|_{\epsilon =0}L(j^1_x
s+\epsilon \omega),
\]
for any $\omega \in T_x ^*N\otimes T_qQ,\,q=s(x)$. The
Euler-Lagrange form is thus
\begin{equation}
\label{canonicalEL}
\mathcal{EL}(L)(j^2s)=\frac{\overline{\nabla}L}{ds}(j^1s)-\mathrm{div}^{\overline{\nabla},\mathbf{v}}\frac{\partial L}{\partial j^{1}s}(j^1s),
\end{equation}
where $\mathrm{div}^{\overline{\nabla},\mathbf{v}}$ stands for the
divergence operator defined by the volume form $\mathbf v$ and the
affine connection $\overline{\nabla}$. It acts on $T^{\ast
}Q$-valued vector fields in $N$ (note that along the map $j^1s$,
$\partial L/\partial j^{1}s$ is precisely a section of $TN\otimes
s^{\ast}T^{\ast }Q\rightarrow N$) and it is defined as the only
operator such that
\begin{align*}
    \mathrm{div}^{\mathbf{v}}\left\langle \mathcal{X},X\right\rangle
=\left\langle \mathrm{div}^{\overline{\nabla},\mathbf{v}}\mathcal{X} ,X\right\rangle +\left\langle
\mathcal{X},\overline{\nabla}X\right\rangle
\end{align*}
for any vector field $\mathcal X\in TN\otimes T^*Q$ and any
section vector field $X$ in $TQ$.

We now assume that the Lagrangian $L:J^{1}(N,Q)=T^{\ast}N\otimes
TQ\rightarrow\mathbb{R}$ can be decomposed as $L=L^{h}+L^{v}$ with
\begin{align*}
        L^{h}(\omega\otimes v)=L^{h}(\omega\otimes p^{h}(v))\qquad\mathrm{and}\qquad  L^{v}
    (\omega\otimes v)=L^{v}(\omega\otimes p^{v}(v))
\end{align*}
for any $\omega\otimes v \in T^{\ast}N \otimes TQ$, with respect
to the connection $\mathcal{A}$. Furthermore, as $TQ=HQ\oplus VQ$,
we have
\begin{align*}
    L^{h}:T^{\ast}N\otimes HQ\rightarrow\mathbb{R}\text{\qquad
    and\qquad} L^{v}:T^{\ast}N\otimes VQ\rightarrow\mathbb{R}.
\end{align*}
Obviously, the $G$ invariance of $L$ and $\mathcal{A}$ extends to
the $G$-invariance of $L^{v}$ and $L^{h}$ so that they drop to the
quotient as
\begin{align*}
    \ell^{h}:J^{1}(N,\Sigma)=T^{\ast}N\otimes
    T\Sigma\rightarrow\mathbb{R\qquad
    }\text{and\qquad}\ell^{v}:T^{\ast}N\otimes\mathfrak{\tilde{g}}\rightarrow \mathbb{R},
\end{align*}
to form the reduce Lagrangian $\ell (j^1 \rho, \sigma) =\ell^{h}(j^1\rho)+\ell^{v}(\rho,
\sigma)$. It is easy to see that
\begin{align*}
    \frac{\delta \ell}{\delta j^1 \rho} = \frac{\delta \ell ^h}{\delta
    j^1 \rho}\qquad\mathrm{and}\qquad \frac{\delta \ell}{\delta \sigma}=\frac{\delta
    \ell ^v}{\delta \sigma}.
\end{align*}

We then consider that the linear connection $\overline{\nabla}$ in $Q$ is
invariant under the action of $G$ so that it projects to a linear
connection $\nabla$ in $\Sigma =Q/G$ by the condition $\nabla _X Y
= \pi_*(\overline{\nabla}_{X^h} Y^h)$. In addition, the connection
$\mathcal{A}$ induces a connection in the associated bundle
$\tilde{\mathfrak{g}}\to \Sigma$. With respect to these
connections we can compute
\begin{align*}
    \frac{\nabla \ell}{d \rho}=\frac{\nabla \ell^h}{d
    \rho}+\frac{\nabla \ell ^v}{d \rho},
\end{align*}
and the Lagrange-Poincar\'e equations \eqref{ff} thus read
\begin{equation}
\left.
\begin{array}{c} \mathrm{div}^{\nabla ,
\mathbf{v}}\left(\dfrac{\delta \ell ^h}{\delta j^1 \rho}\right)
-\dfrac{\nabla \ell ^h}{\delta \rho} = f^h + \dfrac{\nabla \ell
^v}{\delta  \rho} - \left\langle \dfrac{\delta \ell ^v}{\delta
\sigma},i_{T\rho}\bar{\mathcal{B}}\right\rangle , \\
\mathrm{div}^{\mathcal{A}}\dfrac{\delta\ell
^v}{\delta\sigma}+\mathrm{ad}_{\sigma}^{\ast}\dfrac{\delta\ell
^v}{\delta\sigma}=f^{v}.
\end{array}
\right\} \label{ff2}
\end{equation}
The Lagrangian splitting is crucial in this methods and allows the
appearance of the standard Euler-Lagrange equations for $\ell^h$
in the left hand side of the first equation. The second important
ingredient is the force term $f^h$ which will allow us to exactly
obtain the Euler-Lagrange equations by cancelling the right hand
side of the same equation.

\subsection{The un-reduction theorem}

We are now ready to state the central theorem of the un-reduction
method using the field theoretical context described above.
\begin{theorem}
\label{mainth} Let $N$ be a smooth manifold oriented by a volume
form $\mathbf{v}$ and $\pi : Q \to \Sigma$ be a $G$-principal
bundle equipped with a principal connection $\mathcal{A}$. Let
$l:J^1(N,\Sigma)=T^*N\otimes T\Sigma \to \mathbb R$ be a first
order Lagrangian. We consider a $G$-invariant Lagrangian
$L:J^1(N,Q)=T^*N\otimes TQ\to \mathbb{R}$ such that $L=L^h + L^v$
where $L^h\circ p^h = L^h$ is uniquely determined by $l$, $L^v
\circ p^v = L^v$ is arbitrary, and $p^h,p^v$ are the projectors of
the splitting $TQ=HQ\oplus VQ$ induced by $\mathcal{A}$. We also
consider a $G$-equivariant force $F:J^1(N,Q)\to T^*Q$ such that
$F^v=p^v \circ F$ is arbitrary and $F^h = p^h \circ F$ is given by
the condition
\begin{equation}
\label{fh}
f^h = -\frac{\nabla \ell ^v}{\delta \rho} + \left\langle
\frac{\delta \ell ^v}{\delta
\sigma},i_{T\rho}\bar{\mathcal{B}}\right\rangle ,
\end{equation}
for its projection $f^h:J^1(N,\Sigma)\times
(T^*N\otimes\mathfrak{\tilde{g}})\to T^*\Sigma$. Then, the
variational equations of the problem defined by $L$ and $F$ read
\begin{equation}
\label{eqsp} \left.
\begin{array}{c} \mathcal{EL}(L^h)(j^2 s) =0\\
\mathcal{A}^* \mathrm{div}^\mathbf{v}\left(\mathbf J\left(\dfrac{\delta
L^v}{\delta j^1 s}\right)\right)  =  F^v (j^1 s),
\end{array}
\right\}
\end{equation}
where $\mathcal{A}^* : \mathfrak{g}^*\to V^*Q$ is the dual of the
connection form. Finally, critical solutions $s:N\to Q$ of
\eqref{eqsp} project to critical solutions $\rho = [s]_G$ of the
Euler-Lagrange equations $\mathcal{EL}(l)(j^2 \rho)=0$.
\end{theorem}

\begin{proof}
We follow the notations of the preceding sections. The variational
principle of $L$ and $F$ is
\begin{eqnarray*}
0 & = & \delta \int _N L^h\mathbf{v} + \delta \int _N
L^v\mathbf{v} + \int _N \langle F^h ,\delta s\rangle
\mathbf{v} + \int _N \langle F^v ,\delta s\rangle \mathbf{v}\\
&=& \delta \int _N L^h  \mathbf{v} + \int _N \left\langle
\frac{\delta \ell ^v}{\delta \sigma},\delta
\sigma\right\rangle\mathbf{v} + \int _N \left\langle \frac{\nabla
\ell ^v}{\delta \rho},\delta \rho\right\rangle \mathbf{v}\\
&\qquad & + \int _N \left\langle f^h , \delta \rho \right\rangle
\mathbf{v} + \int _N \langle F^v ,\mathcal{A}( \delta s)\rangle
\mathbf{v}\\
&=& \delta \int _N L^h (j^1 s) \mathbf{v} + \int _N \left\langle
\frac{\delta \ell ^v}{\delta \sigma}, \delta \sigma \right\rangle
\mathbf{v} + \int _N \left\langle \frac{\delta \ell ^v}{\delta
\sigma}, \bar{\mathcal{B}}(T\rho , \delta \rho )\right\rangle
\mathbf{v}\\
& \qquad & +\int _N \langle F^v , \mathcal{A}(\delta s)\rangle
\mathbf{v}.
\end{eqnarray*}
From the expression of $\delta \sigma$ in Theorem \ref{LP} with
$\eta (x) = (s(x),\mathcal{A}(\delta s))_G$ we have that
\[
\int _N \left\langle \frac{\delta \ell ^v }{\delta \sigma},\delta
\sigma + \bar{\mathcal{B}}(T\rho , \delta \rho)\right\rangle
\mathbf{v}= \int _N \left\langle \frac{\delta \ell ^v }{\delta
\sigma}, \nabla ^\mathcal{A}\eta -
[\sigma,\eta]\right\rangle\mathbf{v}.
\]
For any $f: N\to \mathfrak{g}$, we recall that the covariant
derivative is $\nabla ^\mathcal{A}(s(x),f(x))_G=(s(x),df(x) +
[\mathcal{A}(j^1s),f])_G=(s(x),df(x))_G +[\sigma, (s(x),f(x))_G]$.
Now, for $f=\mathcal{A}(\delta s)$, we have
\begin{eqnarray*} \int _N \left\langle \frac{\delta
\ell ^v }{\delta \sigma},\delta \sigma + \bar{\mathcal{B}}(T\rho ,
\delta \rho)\right\rangle \mathbf{v}  =  \int _N \left\langle
\frac{\delta \ell ^v }{\delta \sigma}, (s,d\mathcal{A}(\delta
s))_G\right\rangle\mathbf{v} \\
=  \int _N \left\langle \mathbf{J}\left(\frac{\delta L^v}{\delta
j^1 s}\right),d\mathcal{A}(\delta s)\right\rangle \mathbf{v} =
-\int _N \left\langle \mathrm{div}^{\mathbf{v}}\left(\mathbf{J}
\left(\frac{\delta L^v}{\delta j^1
s}\right)\right),\mathcal{A}(\delta s)\right\rangle \mathbf{v}.
\end{eqnarray*}
Finally note that, as $L^h(j^1 s)=l (j^1 \rho)$, the variation of
the action defined by $L^h$ with respect to vertical variations of
$L^h$ automatically vanishes. The variational principle naturally
splits into vertical and horizontal part as equations
(\ref{eqsp}).

Solutions of the variational problem defined by $M$ project to
solutions of the problem defined by $l=\ell ^h$ by Theorem
\ref{LP}.
\end{proof}

\begin{remark}
    If we have $N=\mathbb R$, $\mathbf{v}=dt$, (that is, the case of classical Mechanics)
    we have $\mathrm{div}^\mathbf{v}=d/dt$ and we recover the
    results and equations of \cite{BEBH}.
\end{remark}

The expression of the horizontal force $F^h$ defined by condition
(\ref{fh}) is
\[
F^h=-\frac{\overline{\nabla} L^v}{ds}+\left\langle \mathbf
J\left(\frac{\partial L^v}{\partial j^1s}\right),
i_{Ts}\mathcal{B}\right\rangle .
\]
The variational principle on the un-reduced space of equation
\eqref{un-red-var} is then defined using this particular force
such that the reduced Lagrange-Poincar\'e equations decouples.

The first equation in \eqref{eqsp} is the usual Euler-Lagrange
equation for the horizontal Lagrangian. With respect to the
second, we first note that we cannot exchange the position of
$\mathcal{A}$ and $\mathrm{div}^{\mathbf{v}}$ as the authors do in
\cite{BEBH}. In fact, the divergence of
$\mathcal{A}^*\mathbf{J}(\delta L^v/\delta j^1s)$ would require an
additional (linear) connection in $Q$. Moreover, as we mentioned
in the definition \eqref{momapp} of $\mathbf{J}$ above, we have
that $\mathbf{J}(\delta L^v/\delta j^1s)$ is a covariant momentum
map, so that $\mathrm{div}^{\mathbf{v}}\mathbf{J}(\delta
L^v/\delta j^1s)$ is the expression of a conservation law with
respect to the group of symmetries. The second equation in
\eqref{eqsp} equals this to the vertical part of the force. If one
set $F^v=0$, the conservation law is complete, but sometimes it is
interesting to keep this vertical force as it might be used to
externally control the dynamic along the vertical space.

\subsection{Reconstruction and the surjectivity of the
un-reduction scheme}

The theorem \ref{mainth} above says that solutions of the un-reduced
problem project to solutions of the Euler-Lagrange equations
defined by the Lagrangian $l$. One may ask if this projection is
exhaustive, that is, if every solution of the variational
equations of $l$ is a projection of a solution of $L$. This
question involves some topological constraints concerning $N$ (see
\S \ref{topo}), but we first give an answer assuming that $N$ is
simply connected (or we just consider the question from a local
point of view). From the Lagrange-Poincar\'e reduction theorem,
the variational equations defined by $L$ are equivalent to
\[
\mathcal{EL}(\ell ^h)(j^2\rho)=0,\qquad \mathrm{div} ^\mathcal{A}
\frac{\delta \ell ^v}{\delta \sigma} + \mathrm{ad}^*_\sigma
\frac{\delta \ell ^v}{\delta \sigma} = f^v,
\]
that is, they contain the Euler-Lagrange equations of $l=\ell ^h$
together with an additional set of equations which might restrict
the solution of the first set. They are written in terms of the
map $\sigma : N \to T^*N\otimes \mathfrak{\tilde{g}}$ and $\rho :
N \to \Sigma$. Recall that $\sigma$ determines $\rho$ as $\rho =
\pi _{\tilde{\mathfrak{g}}}\circ \sigma$ (see diagram
(\ref{maps})). The key point is that the first reduced equation
only involves $\rho$ and its first jet $j^1\rho$. Once we have a
solution $\rho$ and $j^1\rho$, we may consider both the second
reduced equation and the compatibility condition. They are now
equations for maps $\sigma$ seen as sections of the bundle
$T^*N\otimes \rho ^* \mathfrak{\tilde{g}}\to N$, which means that
we ``restrict'' the vertical part of our construction to the
fibers which sits only on the solution $\rho$ on the base
manifold.  With the solution of these last equations, we can
perform reconstruction to get a map $s:N\to Q$ such that $\rho =
[s]_G$. Roughly speaking, the reduced equations are uncoupled, so
that $\rho$ and $\sigma$ can be treated separately and the
surjectivity of the un-reduction technique is guaranteed. The
reason of this is the force term which exactly decouples these
equations although it is not explicit in the un-reduced equations.

\subsection{Topological constraints and un-reduction}\label{topo}

The topology of the manifold $N$ may create interesting situations
in the reconstruction and un-reduction frameworks. If $N$ is not
simply connected, the flatness of a connection, that is the
compatibility condition \eqref{comp}, does not ensure the
existence of global integral sections and the surjectivity of the
projection $s \mapsto \rho$ of solutions involves some other
global considerations.

An example of this situation is the following.
Consider $Q=S^3$ and $G=S^1$ so that $Q\to \Sigma = S^2$
is the Hopf fibration. Choose the mechanical connection
$\mathcal{A}$ in this bundle, that is, the connection such that
$H_qS^3 \perp V_q S^3$ with respect to the standard Riemannian
metric in $S^3$. For the sake of simplicity we consider $N= S^1$,
that is, a problem of Mechanics with cyclic solutions where, in
addition, the compatibility condition \eqref{comp} is satisfied
automatically. We denote $\theta$ the coordinate of $S^1$ and we
consider the $G$-invariant Lagrangian $L:J^1(N,S^3)\to\mathbb{R}$,
\[
L(j^1_\theta s)=\frac12\left\Vert \dot{s}(\theta) \right\Vert ^2,
\]
where $\dot{s} = d s /d \theta \in T _{s(\theta)}S^3$, as well as
its decomposition $L=L^h+L^v$ induced by the orthogonal splitting
$\dot{s}(\theta) = \dot{s}^h(\theta) + \dot{s}^v(\theta)$ defined
by the $\mathcal{A}$. The adjoint bundle $\tilde{\mathfrak{g}}\to
S^2$ is a trivial line bundle and the reduced phase space
$J^1(N,\Sigma)\times (T^*N\otimes \mathfrak{\tilde{g}})$ becomes
$TS^2 \times T^*S^1$. We can write the reduced Lagrangian as $\ell
= \ell ^h + \ell ^v$ with
\[
\ell ^h( j^1 \rho) = \frac12\left\Vert \dot{\rho} \right\Vert ^2 ,\qquad
\ell ^v (\sigma )= \frac12 \varsigma^2,
\]
where $\rho : S^1\to \Sigma=S^2$, $\dot{\rho} = d \rho /d \theta$,
and $\sigma = \varsigma d\theta$ with $\varsigma$ a map $S^1\to
\tilde{\mathfrak g}\cong \mathbb R$. The reduced equations are
\[
\nabla \dot{\rho} =0,\qquad \dot{\varsigma} = f^v.
\]
Solutions of the first equation are closed geodesics $\rho$ in
$S^2$. Given one of these, the curves $s(\theta)$ of the
un-reduced problem will be in the restriction $\rho ^* S^3$ of the
Hopf fibration along $\rho$. This restriction is a torus and
according to the reconstruction process seen in \S \ref{CLP}, the
curve $s(\theta)$ must be horizontal with respect to the
connection $\mathcal{A}+\varsigma d\theta$. Under these
circumstances, the curve $s(\theta)$ need not be closed and in
fact, the phase $\varphi \in S^1$ such that $s(2\pi )-s(0)=\varphi
$ is precisely the holonomy of the connection
$\mathcal{A}+\varsigma d\theta$ along the curve $\rho$. The
holonomy of $\mathcal{A}$ alone is $\pi$ (indeed, the connection
$\mathcal{A}$ is not flat and the holonomy is related with the
Chern number of the Hopf bundle, see \cite[Chapter XII]{KN}).
Hence, besides conditions $\dot{\varsigma} =f^v$ and $\varsigma
(2\pi)= \varsigma(0)$, for the closeness of $c(\theta)$ we need
$\varsigma (\theta)$ to satisfy
\[
\int _0 ^{2\pi} \varsigma (\theta)d\theta =-\pi ,
\]
so that we cancel the holonomy of $\mathcal{A}$. Only very
specific functions $f^v$ may accomplish these conditions. For
example, $f^v(\theta)=\mathrm{cos}(\theta)$ gives $\varsigma
(\theta) =\mathrm{sin}(\theta) -1/2$ as a possible solution. Other
functions $f^v$ does not provide closed curves $c(\theta)$.
Furthermore, it is important to note that the constant value of
the holonomy of the fixed connection $\mathcal{A}$ along geodesics
$\rho$ is unusual and other choices of fixed connections
$\mathcal{A}$ will define a holonomy depending on
 $\rho$. In that case, the choice of $f^v$ will depend on the global curve
$\rho$ and will not be a local object.

In other words, there are circumstances where one cannot recover
all solution of the reduced problem from those of the un-reduced
problem. It seems that the freedom in the choice of $L^v$ and,
especially, $F^v$ might solve this issue but their specific
expression will depend on the solution $\rho$ itself. We refer the
reader to \cite{MaMoRa} and \cite{Mont} to some related approaches
to the problem or, for example, \cite{Naka} for a similar
situation to the example above in the context of isoholonomic
problems and quantum computation. The situation for manifolds $N$
of dimension greater than $1$ is, of course, much more
complicated.

\section{Applications}\label{apps-sec}

\subsection{Planar curve matching}
We begin the application section with curve matching, he main motivation of this work, initiated by \cite{BEBH} and extended here to field theories. 

\subsubsection{Geometric setting}

Let $Q= \mathrm{Emb}^+(S^1,\mathbb R^2)$ be the manifold of
positive oriented embeddings from $S^1$ to $\mathbb R^2$. Elements
in $Q$ are maps $c(\theta)\in \mathbb R^2$ for $\theta\in S^1$ and
elements in the tangent space $T_cQ$ are pairs $(c,v)$ with $c\in
\mathrm{Emb}^+(S^1,\mathbb R^2)$ and $u \in C^\infty ( S^1 ,
\mathbb R^2)$ a parametrized vector field along the curve $c$.
Then
\begin{align*}
	TQ = Q \times C^\infty (S^1 , \mathbb{R}^2)
\end{align*}
and we can take a trivial linear connection $\overline{\nabla}$ in
$Q$. We consider an open domain $N\subset\mathbb R\times \mathbb
R$ with the Euclidean metric, coordinates $(t,x)$ and volume form
$\mathbf v=dt\wedge dx$. Elements of the jet bundle
$J^1(N,Q)\simeq T^*N\otimes TQ$ are written as
\begin{align}
    j^1_{(x,t)}c = c_t(\theta)(t,x) dt + c_x(\theta)(t,x)dx,
\end{align}
that is, $c_t$ and $c_x$ are the derivatives of a map $c:N\to Q$
along $t$ and $x$ respectively.

We now consider the group $G= \mathrm{Diff}^+(S^1)$ of orientation
preserving diffeomorphisms of $S^1$ and its Lie algebra $\mathfrak
g= \mathfrak X(S^1)$ which consists of vector fields on $S^1$. The group $G$ acts on the right in
$\mathrm{Emb}^+(S^1,\mathbb R^2)$ as reparametrization of curves
$c$ and the reduced space is the space of shapes in $\mathbb R^2$
\begin{align}
    \Sigma := \frac{Q}{G} = \frac{\mathrm{Emb}^+(S^1,\mathbb
    R^2)}{\mathrm{Diff}^+(S^1)}.
\end{align}
The principal bundle $Q\to \Sigma$ is endowed with a canonical
principal connection $\mathcal{A}$ as follows. Given $u\in T_c Q$,
we consider its tangent and normal decomposition
\[
u(\theta)=v(\theta)\mathbf{t} (\theta) + h(\theta)
\mathbf{n}(\theta ),
\]
where $(\mathbf{t} , \mathbf{n})$ is the orthonormal Frenet frame
along $c$ and $v(\theta), h(\theta)$ scalar functions along the curve.
We clearly have that $v(\theta){\bf t}(\theta)$ is a
vector tangent to the orbits of $G=\mathrm{Diff}^+(S^1)$ so that
$v(\theta){\bf t} (\theta)\in V_cQ$. We can thus define the
horizontal part of $u$ as the part $h(\theta) {\bf n}(\theta )$
and we have a decomposition $TQ=HQ\oplus VQ$.

The definition of a convenient Riemannian metric in
$Q=\mathrm{Emb}^+(S^1,\mathbb R^2)$ invariant with respect to the
action of $G=\mathrm{Diff}^+(S^1)$ is an interesting topic which
has attired the attention of many research works (see, for
example, \cite{BBMM} and \cite{BBM} and the references therein).
The natural $L^2$ metric
\begin{align}
    g(u,v)=\int _{S^1} \langle u(\theta) , v(\theta )\rangle dl,
\end{align}
with $u,v \in T_c Q$, and $dl = | c_\theta|d\theta$ being the
arc-length, is not very useful as it defines a zero geodesic
distance in both $Q$ and $Q/G$. The problem can be overcome in the
shape space $Q/G$ by the metrics
\begin{align}
    g(u,v)=\int _{S^1} (1+A\kappa (\theta)^2) \langle u(\theta) ,v(\theta )\rangle dl,
    \label{weighted-metric}
\end{align}
with $A>0$ and $\kappa$ the Frenet curvature of $c$. But this
metric defines again a zero geodesic distance in $Q$ along
the fibers of the fibration $Q\to Q/G$. A metric with a well
defined Riemannian distance in both $Q$ and $Q/G$ is obtained by
adding higher order derivatives of $u$ and $v$ in a Sobolev-type
expression as
\begin{align}
g(u,v)=\int _{S^1} \left(\langle u(\theta) , v(\theta ) \rangle + A^2\langle D_\theta u(\theta) , D_\theta v(\theta )\rangle \right) dl,
\label{sobolev-metric}
\end{align}
where $D_\theta =
\frac{1}{|c_\theta|}\partial_\theta$ is the arc-length
derivative. We can collect these three cases (as well as many
others, see \cite{BBMM}) as
\begin{align}
    g_{\mathcal P}(u,v) = \int_{S^1} \langle u(\theta),\mathcal P v(\theta)
    \rangle ,
    \label{metric-general}
\end{align}
for a convenient choice of a $G$-invariant self-adjoint
pseudo-differential operator $\mathcal P$ which can depend on the
curve and its derivatives. In particular, the operator for
\eqref{weighted-metric} if $\mathcal{P}=1+A\kappa ^2$, and for
\eqref{sobolev-metric} we have $\mathcal P = 1-A^2D_\theta^2$. One
additional advantage of the operator associated to
\eqref{sobolev-metric} is that it does not depend on the curve,
whereas the operator for \eqref{weighted-metric} depends on the
curvature of the curve where it is evaluated. This represents a
great simplification in the expression of the un-reduced
equations.

\begin{remark}
Even if the mechanical connection $\mathcal A$ in this context of
space of embeddings is easy to visualise and compute, its
structure is far from being trivial. The calculation of its
holonomy and curvature is a whole subject on its own, which should 
be addressed in forthcoming works.
\end{remark}

\subsubsection{Reduction and un-reduction}
Elements of the shape space of plane curves $\Sigma =
\mathrm{Emb}^+(S^1,\mathbb{R}^2)/\mathrm{Diff}^+(S^1)$ are denoted
by $\rho$ and the elements of the jet space
$J^1(N,\Sigma)=T^*N\otimes T\Sigma$ are expressed as
\begin{align*}
    j^1_{(t,x)} \rho = \rho _t (t,x)  dt + \rho _x (t,x) dx.
\end{align*}
Furthermore, elements of $T^*N\otimes \mathfrak{\tilde{g}}$ are
\begin{align*}
    \sigma(t,x) = \sigma _t(t,x) dt + \sigma _x(t,x) dx
\end{align*}
where $\sigma_t (t,x),\sigma_x (t,x)$ belong to the adjoint bundle
$ \tilde{\mathfrak g}\to \Sigma$ and can be understood as vector
fields along a shape $\rho \in \Sigma$ and tangent to it. We
consider the $\mathrm{Diff}^+(S^1)$-invariant Lagrangian
$L:J^1(N,Q)\simeq T^*N\otimes TQ \to \mathbb{R}$
\begin{align}
L(j^1_{(x,t)}c)= \frac12 \int_{S^1}\left (\left \langle c_t , \mathcal Pc_t \right \rangle + \left \langle c_x , \mathcal P c_x \right \rangle\right ) dl
\label{curve-lagrangian}
\end{align}
which can be decomposed as $L=L^h + L^v$ with respect to the connection $\mathcal{A}$ as
\begin{align*}
L^h(j^1_{(x,t)}c) &= \frac12 \int_{S^1}\left (\left \langle h_t , \mathcal Ph_t \right \rangle +  \left \langle h_x , \mathcal P h_x \right \rangle\right ) dl,\\
L^v (j^1_{(x,t)}c) &= \frac12 \int_{S^1}\left (\left \langle v_t ,
\mathcal Pv_t \right \rangle +  \left \langle v_x , \mathcal P v_x
\right \rangle\right ) dl,
\end{align*}
where
\begin{equation*}
        c_t = v_t\mathbf{t} + h_t \mathbf{n} \quad \mathrm{and}\quad c_x= v_x\mathbf{t} + h_x\mathbf{n}.
\end{equation*}

The un-reduction equations \eqref{eqsp} are then computed in the
proposition \ref{ppop} below in the case when $\mathcal P$ is
independent of the curve.
\begin{proposition}\label{ppop}
        The un-reduced equations \eqref{eqsp} for the bi-dimensional problem of planar
simple curves defined by the Lagrangian \eqref{curve-lagrangian}
and the metric \eqref{metric-general} with $\mathcal P$
independent of the curve are
\begin{align}
        \begin{split}
    \partial _{x}\mathcal Ph_{x} + \partial _{t}\mathcal Ph_{t}&= D_\theta(h_{x}\mathcal P v_{x}+h_{t}\mathcal P v_{t})-\kappa H \\
    \partial_x \mathcal P v_x+\partial_t\mathcal P v_t&=F^v
        \end{split}
        \label{un-red-curve}
\end{align}
with the decomposition
\begin{align*}
        c_x =  v_x {\bf t}+ h_x {\bf n},\qquad\qquad c_t =  v_t {\bf t}+ h_t {\bf n}
\end{align*}
and for any choice of vertical force $F^v$, where
\begin{align}
    H=\frac12(h_x\mathcal Ph_x+h_t\mathcal P h_t).
\end{align}

\end{proposition}

\begin{proof}
The Euler-Lagrange equation contains two terms, the first is
readily
    \begin{align*}
        \mathrm{div} \frac{\delta L^h}{\delta j^1c} = \partial_t(\mathcal Ph_t) + \partial_x(\mathcal Ph_x).
    \end{align*}
Before computing the second term of the EL equation, we rewrite
only the temporal part of the Lagrangian in order to simplify the
calculation as
    \begin{align*}
        L^h(c,j^1c)|_t &= \frac12 \int_{S^1} \langle (c_t\cdot \mathbf n)\mathbf n, \mathcal P(c_t\cdot \mathbf n)\mathbf n\rangle dl = \frac12 \int_{S^1} \left(c_t\cdot J\frac{c_\theta}{|c_\theta|}\right)\mathcal P\left(c_t\cdot J\frac{c_\theta}{|c_\theta|}\right)|c_\theta|d\theta.
    \end{align*}
This Lagrangian being horizontal, we just need to consider
variations of $c$ that are horizontal with respect to
$\mathcal{A}$, that is, variations of the form $\delta c= \mathbf
n \xi$, $\xi\in C^\infty (S^1)$. With the identities
$D_\theta\mathbf n = -\kappa \mathbf t$, $J\mathbf n= -\mathbf t$
and $J\mathbf t= \mathbf n$  we compute
    \begin{align*}
        \frac{\partial L^h|_t}{\partial c}\cdot (\mathbf n\xi) &= \int_{S^1} \left(c_t\cdot J(\mathbf n \xi)_\theta\right)\mathcal P\left(c_t\cdot \mathbf n\right)d\theta\\
         &= \int_{S^1} \xi_\theta\left(c_t\cdot J\mathbf n \right)\mathcal P\left(c_t\cdot \mathbf n\right)d\theta +  \int_{S^1} \xi\left(c_t\cdot JD_\theta\mathbf n \right)\mathcal P\left(c_t\cdot \mathbf n\right)dl\\
         &= \int_{S^1} \xi D_\theta\left [ \left(c_t\cdot \mathbf t\right)\mathcal P\left(c_t\cdot \mathbf n\right)\right]dl - \int_{S^1} \xi\kappa \left(c_t\cdot \mathbf n \right)\mathcal P\left(c_t\cdot \mathbf n\right)dl.
    \end{align*}
   Therefore the derivative of the Lagrangian is
   \begin{align*}
           \frac{\partial L^h|_t}{\partial c}= D_\theta(h_t\mathcal Pv_t) - \kappa h_t\mathcal P h_t.
   \end{align*}
   From the symmetry $t\Leftrightarrow x$, the Euler-Lagrange equation follows.
\end{proof}

\begin{remark}
	The term $\kappa H$ can be interpreted as a penalty term for the deformation of most curved regions of the curve. 
	The sign of this term would depend on the concavity or convexity of the curve at this point, and thus this force would try to prevent the curve to be deformed too fast in these regions.
	The equation \eqref{un-red-curve} also shows that the dynamics in $(x,t)$ is governed by the coupling between $h_t$ and $v_t$ required for the shape deformation to be independent of the reparametrisation.
\end{remark}
\begin{remark}
        The un-reduced equations with curvature weighted metric \eqref{weighted-metric} can be computed directly from the variational principle, as in \cite{BEBH}.
The equation will have the same symmetry $x\leftrightarrow t$ but with more complicated terms.
Because this metric is not very useful, the covariant un-reduced equations will not be displayed here.
\end{remark}

\subsection{Horizontal Lagrangians and $\sigma$-models}

The freedom in the choice of forces and Lagrangians in Theorem \ref{mainth} permits the trivial choice of
$L^v=0$ and $F^v=0$. From \eqref{fh}, the horizontal part $F^h$ of
the force automatically vanishes. This simple situation appears
when the un-reduced Lagrangian $L$ is just the pull-back of the
Lagrangian $\ell ^h = l:J^1\Sigma \to \mathbb{R}$ with respect to
the projection $J^1(N,Q)\to J^1 (N,\Sigma)$, $j^1s\mapsto
j^1[s]_G= j^1\rho$. A solution of the problem defined by $L$ is
any map $s:N\to Q$ such that $\rho = [s]_G$ is a solution for $l$.
This means that there is a gauge degeneracy in the sense that,
given a solution $s$ and any map $g:N\to G$, the map
$\bar{s}=s\cdot g$ is also a solution.

Even though these trivial choices for $F$ and $L^v$ are not always
convenient, there are some instances where they appear naturally.
This is the case of $\sigma$-models in homogeneous spaces (see for
example \cite{EF2}, \cite{EF3}, \cite{Guest}, \cite{Dai}). Let
$Q=G$ be a Lie group and $H$ be a closed subgroup such that the
Lie algebra decomposes as $\mathfrak{g}=\mathfrak{m}\oplus
\mathfrak{h}$ for certain vector space $\mathfrak{m}$ such that
$[\mathfrak{h},\mathfrak{m}]\subset\mathfrak{m}$ (that is, we have
a reductive decomposition). We can right translate the
decomposition $\mathfrak{m}\oplus \mathfrak{h} = \mathfrak{g}
=T_eG$ to every $T_gG, g\in G$, thus obtaining a connection
$\mathcal{A}$ for the principal bundle $G\to M$ over the
homogeneous space $\Sigma= M = G/H$. We consider the harmonic, or
$\sigma$-model, problem on maps $\rho : N \to M$ defined by the
Lagrangian
\begin{eqnarray*}
l:J^1(N,M ) & \to & \mathbb{R} \\
j^1\rho & \mapsto & \frac12\left\Vert d\rho \right\Vert ^2 ,
\end{eqnarray*}
where the norm is taken with respect to a pseudo-Riemannian metric
in $N$ and a Riemannian metric in $M$. The lift $L$ of $l$ to
$J^1(N,G)$ is
\begin{eqnarray*}
L:J^1(N,G ) & \to & \mathbb{R}\\
j^1 g & \mapsto & \frac12\left\Vert p^h(dg) \right\Vert ^2 ,
\end{eqnarray*}
where $p^h : TG\to HG$ is the horizontal projection defined by
$\mathcal{A}$ and the norm is taken with respect to the metric of
$N$ and the lift of the metric in $M$ to horizontal vectors in
$M$. Theorem \ref{mainth} can apply and solutions of the
force-free problem defined by $L$ project to the desired harmonic
maps in $M$.

In the majority of the homogeneous spaces where relevant
$\sigma$-models are defined, the group $G$ is endowed with a
bi-invariant metric. In this case, the reductive decomposition is
assumed to be $\mathfrak{m}=\mathfrak{h}^\perp$ and we have a
metric in $M$ by imposing the projection $\pi:G\to M$ to be an
isometric submersion, that is, the metric in $T_xM$ is the same as
the metric in $H_gG$ for any $g$ with $\pi(g)=x$. The group $G$
left-acts on the coset space $M$ by isometries. Hence, the
Lagrangians $l$ and $L$ are both $G$ invariant. This group of
symmetries is too big for $M$ to do reduction (in fact the orbit
space is a single point), but we can perform covariant
Euler-Poincar\'e reduction for $L$. We then get a new reduced
Lagrangian
\begin{eqnarray*}
\bar{l}:J^1(N,G)/G=T^*N\otimes\mathfrak{g} &\to &\mathbb{R}\\
\varsigma & \mapsto & \frac12 \left\Vert \varsigma _\mathfrak{m}
\right\Vert ^2
\end{eqnarray*}
where $\varsigma = \varsigma _\mathfrak{h} + \varsigma
_\mathfrak{m}$ is the splitting defined by the reductive
decomposition. It is easy to see that the Euler-Poincar\'e
equations are
\[
\mathrm{div}^\mathbf{v}\varsigma _\mathfrak{m} + [\varsigma
_\mathfrak{h},\varsigma_\mathfrak{m}]=0
\]
which, together with the suitable compatibility condition, can be
used to get solutions of $L$ that, afterwards, can be projected to
$\Sigma$. This approach is found, for example, in \cite{EF3},
\cite{Dai}, \cite{H}. The advantage of this un-reduction and reduction
procedure relies on the fact that $\mathfrak{g}$ is a simpler
space (is a vector space) than either $G$ and $M$.

The situation can be even put in a more general framework as
follows. Let $L$ be a first order Lagrangian on a Lie group $G$ as
configuration space, which is right invariant under the action of
a subgroup $H$ and left invariant under the group $G$ itself.
Suppose that we are interested in the induced variational problem
in the homogeneous space $G/H$. The un-reduction and reduction
procedure will give first a variational problem in $G$ to finally
induce a problem in the Lie algebra $\mathfrak{g}$ which, in
general, is simpler. See \cite{TiVi} for a description of a
similar situation in Mechanics (that is, $N=\mathbb{R}$).

\subsection{Hyperbolic curvature flow }
The hyperbolic curvature flow of plane curves (see for example
\cite{LeS} or \cite{Wa}) is the variational equation defined by
the Lagrangian
\begin{align*}
L&:T\mathrm{Emb}(S^{1},\mathbb{R}^{2})\rightarrow \mathbb{R},\\
L(c,c_t)&=\int_{S^1}\left( \tfrac{1}{2}\left\Vert c_{t}\right\Vert ^{2}-1\right) dl.
\end{align*}
Note that this is not a geodesic variational principle of the
$L^2$ metric (which provides null geodesic distances in both the
curve and shape spaces) but a Lagrangian involving a kinetic and a
potential term. Moreover, the Lagrangian $L$ can be easily split
into horizontal and vertical with respect to the connection
$\mathcal{A}(c_t)\mathcal{=(}c_{t}\cdot {\bf t}){\bf t}$ as
\begin{align*}
L^{h}=\int_{S^1}\left( \tfrac{1}{2}h^{2}-1\right) dl,\qquad
L^{v}=\int_{S^1}\tfrac{1}{2}v^{2}dl,
\end{align*}
where
\begin{align*}
c_t = h {\bf n} + v {\bf t}.
\end{align*}
The Lagrangian $L$ (and $L^h$, $L^v$) is
$\mathrm{Diff}(S^1)$-invariant as its definition is geometric and
does not depend on the parametrization of $c$ but only on its shape.

One of the main features and applications of the hyperbolic flow
(as well as of other geometric flows of curves) is the study of
the evolution of the shapes of the curves under it. We can now
suppose that we just want to study this evolution in the shape
space $\mathrm{Emb}(S^{1},\mathbb{R}^{2})/\mathrm{Diff}^+(S^1)$.
The natural Lagrangian in this situation becomes $l=\ell ^h$, the
projection of $L^h$ to this quotient space. In this context, the
un-reduction technique applies and we have the last result of this paper.

\begin{proposition}
The un-reduced equations for the system described above read
\begin{eqnarray*}
\partial_t h &=&D_\theta(vh)-\kappa (\tfrac{1}{2}h^{2}-1), \\
\partial_t v &=&F^{v}, \\
c_{t} &=& h{\bf n}+v {\bf t}.
\end{eqnarray*}
In particular, if we choose $F^v=0$ and the initial tangent
velocity to vanish ($v(0)=0$), then $v(t)=0$ for all times and the
velocity of $h$ is proportional to the curvature $\kappa$.

\end{proposition}

\begin{proof}
        Following the derivation of the un-reduced equation for curve matching, but in the classical case, one can prove this proposition as well.
\end{proof}

\begin{remark}
The equations of the Proposition above for $F^v=0$ are the
hyperbolic mean flow equations (see for example \cite{LeS} for a
good account of this flow). The usual approach in the literature
works in $\mathrm{Emb}(S^1 ,\mathbb{R}^2)$ and then restrict
oneself to the normal part of the flow. The approach here works
with shapes in $\mathrm{Emb}(S^1 \mathbb{R}^2)/\mathrm{Diff}(S^1)$
so that the trivial choice of $F^v=0$ gives directly the geometric
equations.
\end{remark}

\section*{Acknowledgements}
We are grateful to M. Bauer, S. Durrleman, R. Montgomery and T.
Ratiu for valuable discussions during the course of this work. 
We also want to thank H.Dumpty for helpful suggestions during the 
realisation of Fig.~\ref{fig:humpty}.  
AA acknowledges partial support from an Imperial College 
London Roth Award, AA and DH from the European Research 
Council Advanced Grant 267382 FCCA. MCL has been partially 
funded by MINECO (Spain) under projects MTM2011-22528 and 
MTM2010-19111. MCL wants to thank Imperial College for 
its hospitality during the visit in which the main ideas 
of this work were developed.

\end{document}